\documentclass[10.5pt]{article}
\usepackage{amsmath}
\usepackage{amsfonts}
\usepackage{graphicx}
\usepackage{booktabs}
\usepackage{amsmath}%\numberwithin{equation}{section}
\usepackage{amssymb}
\usepackage{latexsym}
\usepackage{amsmath,amsfonts,amssymb,amsthm,euscript,makeidx,color,mathrsfs}
 \usepackage{float}
\usepackage{algorithm}
\usepackage{algpseudocode}

\renewcommand{\theequation}{\arabic{section}.\arabic{equation}}
\renewcommand{\baselinestretch}{1.0}
\newtheorem{problem}{Problem}
\newtheorem{theorem}{Theorem}%[section]
\newtheorem{lemma}{Lemma}%[section]
\newtheorem{remark}{Remark}%[section]
\renewcommand{\theequation}{\arabic{equation}}

\makeatletter
   \renewcommand{\theequation}{%
            \thesection.\arabic{equation}}
   \@addtoreset{equation}{section}
\makeatother

\title{A General Method for Optimal Decentralized Control with Current State/Output Feedback Strategy  }

\author{
Hongdan Li,\ Yawen Sun and\ Huanshui Zhang
\thanks{This work was supported by the Original Exploratory Program Project of National Natural Science Foundation of China (62450004), the Major Basic Research of Natural Science Foundation of Shandong Province (ZR2021ZD14), the National Natural Science Foundation of China (62103241,62273213), Shandong Provincial Natural Science Foundation (ZR2021QF107).}% <-this % stops a space
\thanks{H. Li, and Y. Sun are with College of Electrical Engineering and Automation, Shandong University of Science and Technology, Qingdao, Shandong, P.R.China 266590. H. Zhang is with the College of Electrical Engineering and Automation, Shandong University of Science and Technology, Qingdao 266590, China, and also with the School of Control Science and Engineering, Shandong University, Jinan, Shandong 250061, China (e-mail: hszhang@sdu.edu.cn). }% <-this % stops a space}
}

\begin{document}

\pagenumbering{arabic}
 \setcounter{page}{1}

\pagenumbering{arabic} \thispagestyle{empty} \setcounter{page}{1}

%{\footnotesize \noindent
\baselineskip 16pt
\date{}
\maketitle
\begin{abstract}
This paper explores the decentralized control of linear deterministic systems in which different controllers operate based on distinct state information, and extends the findings to the output feedback scenario. Assuming the controllers have a linear state feedback structure, we derive the expression for the controller gain matrices using the matrix maximum principle. This results in an implicit expression that couples the gain matrices with the state. By reformulating the backward Riccati equation as a forward equation, we overcome the coupling between the backward Riccati equation and the forward state equation.  Additionally, we employ a gradient descent algorithm to find the solution to the implicit equation. This approach is validated through simulation examples.
\end{abstract}

\section{Introduction}

Optimal control theory is essential to modern control systems, forming the foundation of the entire field \cite{Bryson75}, \cite{Lewis12}, \cite{Zhou96}. Among its various applications, linear quadratic (LQ) optimal control is particularly prominent and representative in linear system synthesis. Initially, research in this area was mainly directed at systems managed by a single controller or multiple controllers with shared information, which could be effectively addressed through augmentation techniques. These centralized control problems were typically resolved using classical approaches.

Decentralized control is a control system design method in which multiple independent controllers work together to achieve the overall control objectives of the system, rather than relying on a single central controller. Each controller makes decisions based on local information it obtains and typically does not directly exchange information with other controllers. Decentralized control systems are widely applied in complex systems such as multi-vehicle formations, smart grids, and large-scale industrial process control. Decentralized control has garnered significant attention, with early research focusing primarily on static aspects like the stochastic team decision problem studied by Marschak \cite{Marschak58} and Radner \cite{Radner62}. Witsenhausen \cite{Witsenhausen68} expanded this research to dynamic cases, introducing the famous \emph{Witsenhausen counterexample}, which demonstrated that optimal solutions for decentralized LQ control could be nonlinear under non-classical information patterns. Research into discrete-time decentralized problems explored the separation of estimation and control, highlighting the one-step delayed sharing pattern \cite{Witsenhausen71}. Additionally, optimal control of LQG systems with this information pattern was studied in \cite{Kurtaran74}. Ba\c{s}ar \cite{Basar78} further contributed by providing unique affine noncooperative equilibrium solutions for multistage LQG decision problems, specifically focusing on scenarios with two decision makers and one-step delayed observation sharing.

Regarding other information structures, various studies have examined decentralized optimal control under different conditions. For instance, \cite{Lessard15} and \cite{Nayyar15} explored partially nested information structures. In cases with partial historical sharing, where controllers share some historical information, Mahajan and Nayyar \cite{Mahajan15} derived optimal linear control strategies using the common information approach. Liang \emph{et al.} \cite{Liang22}, under the assumption of linear control strategies and utilizing the maximum principle, studied decentralized control for networked control systems with partial historical sharing. They considered scenarios where the remote and local controllers have access to different information and presented optimal estimators for both controllers based on asymmetric observations. More recent work by Xu \emph{et al.} \cite{Xu23} investigated a class of discrete-time decentralized LQG control problems involving two controllers and a d-step delayed information sharing pattern, providing explicit forms for the pair of optimal linear controllers. For additional research on nonclassical information structures, references \cite{Asghari15}, \cite{Chang11}, \cite{Rotkowitz06}, and others offer further insights.

It is worth noting that the aforementioned literature focuses on scenarios where all measurement information is corrupted by noise. Consequently, filtering is incorporated into the designed controllers. However, in many control and regulator problems, a common and practical assumption is the availability of ``noise-free" measurements. Therefore, it is necessary to further consider the deterministic case with asymmetric information. It has important applications in deterministic control fields such as process control, robotic systems, traffic management, and so on. This paper investigates decentralized control of linear deterministic systems where different controllers are based on different state information. Under the assumption that the controllers have a linear state feedback structure, we derive the expression for the controller gain matrices by applying the matrix maximum principle. This is an implicit expression that couples the gain matrices with the state. The coupling between the backward Riccati equation and the forward state equation makes solving this equation very challenging. Therefore, we consider the algebraic Riccati equation in the steady-state case and solve it by reformulating the backward Riccati equation as a forward equation. Furthermore, we use a gradient descent algorithm to find the solution to the implicit equation. This approach is validated through simulation examples.

This paper can be broadly divided into the following sections. The second section describes the problem under study and the related preparatory work. In the third section, we present the main results and the corresponding algorithm design. The fourth section demonstrates the effectiveness of the algorithm by providing a four-dimensional simulation example. Finally, we conclude the paper with a summary.

\emph{\textbf{Notation}:} ${\mathbb{R}}^n$ is the $n$-dimensional Euclidean space. $Y'$ is  the transpose of $Y$ and  $Y\geq 0 \ (Y>0)$ means  that  $Y$ is symmetric positive semi-definite  (positive definite).

\section{Problem  Statement and Preliminary}
\normalsize
\subsection{Problem  Statement}
Consider the following discrete-time linear system
\begin{eqnarray}
x(k+1)\hspace{-3mm}&=&\hspace{-3mm}Ax(k)+Bu(k), \label{f2.1}
\end{eqnarray}
where $x\in \mathbb{R}^{n},  u\in \mathbb{R}^{m}$ are the state and input, respectively. $A, B$ are constant matrices with suitable dimensions.

Assume that the controller has the following structure
\begin{eqnarray}
u(k)\hspace{-3mm}&=&\hspace{-3mm}K(k)x(k)=\begin{bmatrix}K_{1}(k)&0\\0&K_{2}(k)\end{bmatrix}x(k),
\label{f2.2}
\end{eqnarray}
where $K_{i}(k)$ are the gain matrices to be determined.

Define the following quadratic cost functional
\begin{eqnarray}
J_{N}\hspace{-3mm}&=&\hspace{-3mm}\sum^{N}_{k=0}[x(k)'Qx(k)+u(k)'Ru(k)]+x(N+1)'P_{N+1}x(N+1),\label{f2.3}
\end{eqnarray}
where the weighting matrices $Q, R$ and $P_{N+1}\geq0.$
%And the terminal value $P_{N+1}\geq0$.

\begin{problem}
Find the gain matrices $K_{1}$ and $K_{2}$ to minimize the cost functional (\ref{f2.3}) subject to system equation (\ref{f2.1}) and linear structure (\ref{f2.2}).
\end{problem}

\begin{remark}
Compared with the classical LQ, it is essentially a decentralized control problem based on private information. As a result, the classical control theory cannot be directly applied. Furthermore, since the controllers are based only on their private state information and do not share it, the dimensional expansion method fails.
\end{remark}

\subsection{Preliminary}

For discussing, from (\ref{f2.2}), let $x(k)=\begin{bmatrix}x_{1}(k)\\x_{2}(k)\end{bmatrix}$. System (\ref{f2.1}) with (\ref{f2.2}) can be rewritten
\begin{eqnarray}
\left\{
  \begin{array}{ll}
   x_{1}(k+1)=(A_{11}+B_{11}K_{1}(k))x_{1}(k)
+(A_{12}+B_{12}K_{2}(k))x_{2}(k),\\
   x_{2}(k+1)=(A_{21}+B_{21}K_{1}(k))x_{1}(k)
+(A_{22}+B_{22}K_{2}(k))x_{2}(k).
  \end{array}
\right.
\label{f2.4}
\end{eqnarray}

Considering that the essence of the problem is to find the optimal gain matrices $K_{1}$ and $K_{2}$, we adopt the matrix maximum principle for the solution. Therefore, let %$X(k+1)=x(k+1)x(k+1)'$,
$X_{11}(k+1)=x_{1}(k+1)x_{1}(k+1)'$, $X_{12}(k+1)=x_{1}(k+1)x_{2}(k+1)'$, $X_{22}(k+1)=x_{2}(k+1)x_{2}(k+1)'$.
Accordingly,  we have
\begin{small}\begin{eqnarray}
%X(k+1)\hspace{-3mm}&=&\hspace{-3mm}(A+BK(k))X(k)(A+BK(k))',\nonumber\\
X_{11}(k+1)
\hspace{-3mm}&=&\hspace{-3mm}(A_{11}+B_{11}K_{1}(k))X_{11}(k)(A_{11}+B_{11}K_{1}(k))'
+(A_{11}+B_{11}K_{1}(k))X_{12}(k)(A_{12}+B_{12}K_{2}(k))'\nonumber\\
\hspace{-3mm}&&\hspace{-3mm}
+(A_{12}+B_{12}K_{2}(k))X_{12}(k)'(A_{11}+B_{11}K_{1}(k))'
+(A_{12}+B_{12}K_{2}(k))X_{22}(k)(A_{12}+B_{12}K_{2}(k))',\label{x11}\\
X_{12}(k+1)
\hspace{-3mm}&=&\hspace{-3mm}(A_{11}+B_{11}K_{1}(k))X_{11}(k)(A_{21}+B_{21}K_{1}(k))'
+(A_{11}+B_{11}K_{1}(k))X_{12}(k)(A_{22}+B_{22}K_{2}(k))'\nonumber\\
\hspace{-3mm}&&\hspace{-3mm}
+(A_{12}+B_{12}K_{2}(k))X_{12}(k)'(A_{21}+B_{21}K_{1}(k))'
+(A_{12}+B_{12}K_{2}(k))X_{22}(k)(A_{22}+B_{22}K_{2}(k))',\label{x12}\\
X_{22}(k+1)
\hspace{-3mm}&=&\hspace{-3mm}(A_{21}+B_{21}K_{1}(k))X_{11}(k)(A_{21}+B_{21}K_{1}(k))'
+(A_{21}+B_{21}K_{1}(k))X_{12}(k)(A_{22}+B_{22}K_{2}(k))'\nonumber\\
\hspace{-3mm}&&\hspace{-3mm}
+(A_{22}+B_{22}K_{2}(k))X_{12}(k)'(A_{21}+B_{21}K_{1}(k))'
+(A_{22}+B_{22}K_{2}(k))X_{22}(k)(A_{22}+B_{22}K_{2}(k))'.\label{x22}
\end{eqnarray}\end{small}

From (\ref{f2.1}) and (\ref{f2.2}), the cost functional (\ref{f2.3}) can be rewritten as follows.
\begin{eqnarray}
J_{N}
\hspace{-3mm}&=&\hspace{-3mm}\sum^{N}_{k=0}\{x_{1}(k)'[Q_{11}+K_{1}(k)'R_{11}K_{1}(k)]x_{1}(k)
+2x_{1}(k)'[Q_{12}+K_{1}(k)'R_{12}K_{2}(k)]x_{2}(k)\nonumber\\
\hspace{-3mm}&&\hspace{-3mm}
+x_{2}(k)'[Q_{22}+K_{2}(k)'R_{22}K_{2}(k)]x_{2}(k)\}+x'(N+1)P_{N+1}x(N+1)\nonumber\\
\hspace{-3mm}&=&\hspace{-3mm}tr\sum^{N}_{k=0}\{[Q_{11}+K_{1}(k)'R_{11}K_{1}(k)]X_{11}(k)
+2[Q_{12}+K_{1}(k)'R_{12}K_{2}(k)]X_{12}(k)\nonumber\\
\hspace{-3mm}&&\hspace{-3mm}
+[Q_{22}+K_{2}(k)'R_{22}K_{2}(k)]X_{22}(k)\}
,\label{2.5}
\end{eqnarray}
where $Q=\begin{bmatrix}Q_{11}&Q_{12}\\Q_{12}'&Q_{22}\end{bmatrix}$, and $R=\begin{bmatrix}R_{11}&R_{12}\\R_{12}'&R_{22}\end{bmatrix}$.

\section{Main Result}
\subsection{Main Theorem}
Based on the preparatory work in the aforementioned section, we now give the main result.
\begin{theorem}
The solution of Problem 1 is given by
\begin{eqnarray}
\Upsilon_{11}(k)K_{1}(k)X_{11}(k)+M_{11}(k)X_{11}(k)+\Upsilon_{12}(k)K_{2}(k)X_{12}(k)'
+M_{12}(k)X_{12}(k)'&=&0,\label{3.1}\\
\Upsilon_{22}(k)K_{2}(k)X_{22}(k)+M_{22}(k)X_{22}(k)+\Upsilon_{12}(k)'K_{1}(k)X_{12}(k)
+M_{21}(k)X_{12}(k)&=&0, \label{3.2}
\end{eqnarray}
where $X_{ij}(k),\ i,j=1,2$ satisfying (\ref{x11})-(\ref{x22}), and
\begin{eqnarray*}
\Upsilon(k)&=&\begin{bmatrix}\Upsilon_{11}(k)&\Upsilon_{12}(k)\\ \Upsilon_{12}(k)'&\Upsilon_{22}(k)\end{bmatrix}=R+B'P(k+1)B,\label{3.3}\\
M(k)&=&\begin{bmatrix}M_{11}(k)&M_{12}(k)\\ M_{21}(k)&M_{22}(k)\end{bmatrix}=B'P(k+1)A \label{3.4}
\end{eqnarray*}
with
\begin{eqnarray}
P(k)
%&=&\begin{bmatrix}P_{11}(k)&P_{12}(k)\\ P_{12}(k)'&P_{22}(k)\end{bmatrix}
=Q+K(k)'RK(k)+(A+BK(k))'P(k+1)(A+BK(k)),\label{3.5}
\end{eqnarray}
whose terminal value is $P_{N+1}$.

%In this case, the optimal cost function value is $J^{\ast}_{N}=x_{0}'P_{0}x_{0}$.
%\begin{small}\begin{eqnarray}
%\Upsilon_{11}(k)&=&\begin{bmatrix}I&0\end{bmatrix}\Upsilon(k)\begin{bmatrix}I\\0\end{bmatrix}\\
%M_{11}(k)&=&\begin{bmatrix}I&0\end{bmatrix}M(k)\begin{bmatrix}I\\0\end{bmatrix},\\
%\Upsilon_{12}(k)&=&\begin{bmatrix}I&0\end{bmatrix}\Upsilon(k)\begin{bmatrix}0\\I\end{bmatrix},\\
%M_{12}(k)&=&\begin{bmatrix}I&0\end{bmatrix}M(k)\begin{bmatrix}0\\I\end{bmatrix}
%\\
%\Upsilon_{22}(k)&=&\begin{bmatrix}0&I\end{bmatrix}\Upsilon(k)\begin{bmatrix}0\\I\end{bmatrix},\\
%M_{22}(k)&=&\begin{bmatrix}0&I\end{bmatrix}M(k)\begin{bmatrix}0\\I\end{bmatrix},\\
%M_{21}(k)&=&\begin{bmatrix}0&I\end{bmatrix}M(k)\begin{bmatrix}I\\0\end{bmatrix},\\
%X(k+1)
%\hspace{-3mm}&=&\hspace{-3mm}(A+BK(k))X(k)(A+BK(k))'.
%\end{eqnarray}\end{small}
\end{theorem}

\begin{proof}
From (\ref{2.5}), we introduce the following Hamiltonian functional with matrix $P(k+1)$.
\begin{eqnarray*}
H_{N}\hspace{-3mm}&=&\hspace{-3mm}tr[(Q+K(k)'RK(k))X(k)+(A+BK(k))X(k)(A+BK(k))'P(k+1)]
%\hspace{-3mm}&=&\hspace{-3mm}tr\sum^{N}_{k=0}\Big\{[Q_{11}+K_{1}(k)'R_{11}K_{1}(k)]X_{11}(k)
%+2[Q_{12}+K_{1}(k)'R_{12}K_{2}(k)]X_{12}(k)'
%+[Q_{22}+K_{2}(k)'R_{22}K_{2}(k)]X_{22}(k)\\
%\hspace{-3mm}&&\hspace{-3mm}+P_{11}(k+1)[(A_{11}+B_{11}K_{1}(k))X_{11}(k)(A_{11}+B_{11}K_{1}(k))'
%+(A_{11}+B_{11}K_{1}(k))X_{12}(k)(A_{12}+B_{12}K_{2}(k))'\nonumber\\
%\hspace{-3mm}&&\hspace{-3mm}
%+(A_{12}+B_{12}K_{2}(k))X_{12}(k)'(A_{11}+B_{11}K_{1}(k))'
%+(A_{12}+B_{12}K_{2}(k))X_{22}(k)(A_{12}+B_{12}K_{2}(k))'-X_{11}(k+1)]\nonumber\\
%\hspace{-3mm}&&\hspace{-3mm}+2P_{12}(k+1)[(A_{11}+B_{11}K_{1}(k))X_{11}(k)(A_{21}+B_{21}K_{1}(k))'
%+(A_{11}+B_{11}K_{1}(k))X_{12}(k)(A_{22}+B_{22}K_{2}(k))'\nonumber\\
%\hspace{-3mm}&&\hspace{-3mm}
%+(A_{12}+B_{12}K_{2}(k))X_{12}(k)'(A_{21}+B_{21}K_{1}(k))'
%+(A_{12}+B_{12}K_{2}(k))X_{22}(k)(A_{22}+B_{22}K_{2}(k))'-X_{12}(k+1)]\nonumber\\
%\hspace{-3mm}&&\hspace{-3mm}+P_{22}(k+1)[(A_{21}+B_{21}K_{1}(k))X_{11}(k)(A_{21}+B_{21}K_{1}(k))'
%+(A_{21}+B_{21}K_{1}(k))X_{12}(k)(A_{22}+B_{22}K_{2}(k))'\nonumber\\
%\hspace{-3mm}&&\hspace{-3mm}
%+(A_{22}+B_{22}K_{2}(k))X_{12}(k)'(A_{21}+B_{21}K_{1}(k))'
%+(A_{22}+B_{22}K_{2}(k))X_{22}(k)(A_{22}+B_{22}K_{2}(k))'-X_{22}(k+1)]\Big\}
\end{eqnarray*}
Due to the special information structure of the gain matrix $K(k)$, differentiating the Hamiltonian function yields
\begin{eqnarray*}
0=\frac{\partial H_{N}}{\partial K_{1}(k)}
\hspace{-3mm}&=&\hspace{-3mm}[R_{11}+B_{11}'P_{11}(k+1)B_{11}
+B_{21}'P_{12}(k+1)B_{11}+B_{11}'P_{12}(k+1)B_{21}\\
\hspace{-3mm}&&\hspace{-3mm}
+B_{21}'P_{22}(k+1)B_{21}]K_{1}(k)X_{11}(k)
+[B_{11}'P_{11}(k+1)A_{11}+B_{21}'P_{12}(k+1)A_{11}\\
\hspace{-3mm}&&\hspace{-3mm}+B_{11}'P_{12}(k+1)A_{21}+B_{21}'P_{22}(k+1)A_{21}]X_{11}(k)
+[R_{12}K_{2}(k)+B_{11}'P_{11}(k+1)(A_{12}\\
\hspace{-3mm}&&\hspace{-3mm}+B_{12}K_{2}(k))+B_{11}'P_{12}(k+1)(A_{22}+B_{22}K_{2}(k))
+B_{21}'P_{22}(k+1)(A_{22}+B_{22}K_{2}(k))\\
\hspace{-3mm}&&\hspace{-3mm}+B_{21}'P_{12}(k+1)(A_{12}+B_{12}K_{2}(k))]X_{12}(k)'\\
\hspace{-3mm}&=&\hspace{-3mm}\Upsilon_{11}(k)K_{1}(k)X_{11}(k)+M_{11}(k)X_{11}(k)+\Upsilon_{12}(k)K_{2}(k)X_{12}(k)'
+M_{12}(k)X_{12}(k)',\\
0=\frac{\partial H_{N}}{\partial K_{2}(k)}
\hspace{-3mm}&=&\hspace{-3mm}[R_{22}+B_{12}'P_{11}(k+1)B_{12}
+B_{12}'P_{12}(k+1)B_{22}+B_{22}'P_{12}(k+1)B_{12}\\
\hspace{-3mm}&&\hspace{-3mm}+B_{22}'P_{22}(k+1)B_{22}]K_{2}(k)X_{22}(k)
+[B_{12}'P_{11}(k+1)A_{12}+B_{12}'P_{12}(k+1)A_{22}\\
\hspace{-3mm}&&\hspace{-3mm}+B_{22}'P_{12}(k+1)A_{12}
+B_{22}'P_{22}(k+1)A_{22}]X_{22}(k)
+[R_{12}K_{1}(k)+B_{12}'P_{11}(k+1)(A_{11}\nonumber\\
\hspace{-3mm}&&\hspace{-3mm}+B_{11}K_{1}(k))+B_{22}'P_{12}(k+1)(A_{11}+B_{11}K_{1}(k))
+B_{12}'P_{12}(k+1)(A_{21}+B_{21}K_{1}(k))\nonumber\\
\hspace{-3mm}&&\hspace{-3mm}+B_{22}'P_{22}(k+1)
(A_{21}+B_{21}K_{1}(k))]X_{12}(k)
\nonumber\\
\hspace{-3mm}&=&\hspace{-3mm}\Upsilon_{22}(k)K_{2}(k)X_{22}(k)+M_{22}(k)X_{22}(k)+\Upsilon_{12}(k)'K_{1}(k)X_{12}(k)
+M_{21}(k)X_{12}(k),\\
P(k)=\frac{\partial H_{N}}{\partial X(k)}
\hspace{-3mm}&=&\hspace{-3mm}Q+K(k)'RK(k)+(A+BK(k))'P(k+1)(A+BK(k)),
\end{eqnarray*}
which are the solution of Problem 1.

%For giving the optimal cost value, we define $V_{k}=x(k)'P(k)x(k)$. Thus, we have
%\begin{eqnarray*}
%V_{k}-V_{k+1}\hspace{-3mm}&=&\hspace{-3mm}
%x(k)'P(k)x(k)-(Ax(k)+Bu(k))P(k+1)(Ax(k)+Bu(k))\\
%\hspace{-3mm}&=&\hspace{-3mm}
%x(k)'[P(k)-(A+BK(k))'P(k+1)(A+BK(k))]x(k)\\
%\hspace{-3mm}&=&\hspace{-3mm}
%x(k)'[Q+K(k)'RK(k)]x(k).
%\end{eqnarray*}
%By summing both sides of the above equation from $k=0$ to $N$, when $K(k)$ satisfy (\ref{3.1}) and (\ref{3.2}),  we obtain the optimal cost value as follows.
%\begin{eqnarray*}
%J^{\ast}_{N}=V_{0}=x_{0}'P(0)x_{0}.
%\end{eqnarray*}
\end{proof}

\begin{remark}
The above result provides the expression for $K_{i}(k)$, which is in the form based on the forward state Lyapunov equation and the backward Riccati equation. It is very challenging to solve it directly. Considering that in the steady-state form of the Riccati equation (i.e., the algebraic Riccati equation), the results obtained through forward sequence approximation and backward approximation are consistent, we solve it by transforming the backward Riccati equation into a forward equation. It makes the aforementioned equations (\ref{3.1}) and (\ref{3.2}) solvable. The specific algorithm is presented below.
\end{remark}

\subsection{Algorithm Design}
From the above main results, it can be seen that solving Problem 1 hinges on solving the forward and backward matrix equations (\ref{3.1})-(\ref{3.2}) and (\ref{3.5}), which is difficult to solve. Inspired by solving algebraic equations under steady-state conditions, we transform it into solving forward iterative equations. That is, transforming the backward equation (\ref{3.5}) into a forward iteration with given initial conditions.
\begin{eqnarray}
P(k+1)
%&=&\begin{bmatrix}P_{11}(k)&P_{12}(k)\\ P_{12}(k)'&P_{22}(k)\end{bmatrix}
=Q+K(k)'RK(k)+(A+BK(k))'P(k)(A+BK(k))\label{3.05}
\end{eqnarray}
with the terminal value $P(0)=\delta I \ (\delta>0)$.

For convenience of designing the following algorithm, the following lemma is given.
\begin{lemma}\cite{Graham18}\label{l1}
For any three matrices $A, B$ and $C$ of suitable dimension,
\begin{eqnarray*}
vec(ABC)=(C'\otimes A)vec(B).
\end{eqnarray*}
\end{lemma}

According to Lemma \ref{l1}, equations (\ref{3.1}) and (\ref{3.2}) can be rewritten as follows.
\begin{align*}
\bar{A}(k)\bar{x}(k)=\bar{b}(k),
\end{align*}
where
\begin{align*}
\bar{A}(k)={}&\left[\begin{array}{cccc}
	X_{11}(k)'\otimes \Upsilon_{11}(k)
        & X_{12}(k)\otimes \Upsilon_{12}(k)     \cr
	X_{12}(k)'\otimes \Upsilon_{12}(k)'
        & X_{22}(k)\otimes \Upsilon_{22}(k)    \cr
\end{array}\right],
\bar{x}(k)=vec\left[\begin{array}{cccc}
	K_1(k)           \cr
	K_2(k)     \cr
\end{array}\right]\\
\bar{b}(k)={}&-vec\left[\begin{array}{cccc}
	M_{11}(k)X_{11}(k)+M_{12}(k)X_{12}(k)'           \cr
	M_{22}(k)X_{22}(k)+M_{21}(k)X_{12}(k)     \cr
\end{array}\right].
\end{align*}

Now, let
\begin{align*}
f(\bar{x}(k)) = \left(\bar{A}(k)\bar{x}(k)-\bar{b}(k)\right)'\left(\bar{A}(k)\bar{x}(k)-\bar{b}(k)\right).
\end{align*}
Obviously, the above function's twice differentiability with respect to $\bar{x}$ can be ensured. Our goal is to find the optimal $\bar{x}$ to minimize $f$, that is, $\min\limits_{\bar{x}}f(\bar{x}(k))$. Actually, for every time $k$, the following optimal control problem can be established
\begin{eqnarray}
&&\min\limits_{\bar{x}}\sum^{I}_{i=0}[f(\bar{x}^{i}(k))+\frac{1}{2}v^{i}(k)'\bar{R}v^{i}(k)]+f(\bar{x}^{I+1}(k)), \label{a1}\\
&& subject \ \ to\ \ \bar{x}^{i+1}(k)=\bar{x}^{i}(k)+v^{i}(k), \label{a2}
\end{eqnarray}
where $v^{i}(k)$ is the control of system (\ref{a2}), and integer $I>0$ is the control terminal time. The weighting matrix $\bar{R}$ is positive. According to the approach proposed in the literature \cite{Wang24}, the algorithm for solving equations (\ref{3.1}) and (\ref{3.2}) is given as follows.
%\begin{align}
%min\sum_{i=0}^I\left[f(K^i(k))+\frac{1}{2}u^T(i)\bar{R}u(i)\right]+f(K^I(k)),
%\end{align}
%where
%$u(i)=-g_{i}(K^i(k))$

\begin{algorithm}[htb]
	\renewcommand{\thealgorithm}{}
	\caption{\hspace{-0.6em} \textbf{1} Numerical algorithm for solving Problem 1.}
	\label{alg:Framwork}
	\begin{algorithmic}[1]
\State \textbf{Initialization:}  $\bar{x}^0(0)$, $\bar{R}$, $I$;
\For{$i=0$ \textbf{to} $I$}
\State $g_0(\bar{x}^i(0))=(\bar{R}+f^{(2)}(\bar{x}^i(0)))^{-1}f^{(1)}(\bar{x}^i(0))$;
\For{$j=1$ \textbf{to} $i$}
\State $g_{j}(\bar{x}^i(0))=(\bar{R}+f^{(2)}(\bar{x}^i(0)))^{-1}[f^{(1)}(\bar{x}^i(0))+\bar{R}g_{j-1}(\bar{x}^i(0))];$
\EndFor
\State $\bar{x}^{i+1}(0)=\bar{x}^{i}(0)-g_{i}(\bar{x}^i(0));$
\EndFor
\State $\bar{x}(0)=\bar{x}^{I}(0)$;
\For{$k=1$ \textbf{to} $N$}
\State  \textbf{Initialization:} $\bar{x}^0(k)=\bar{x}(k-1)$;
   \For{$i=0$ \textbf{to} $I$}
\State $g_0(\bar{x}^i(k))=(\bar{R}+f^{(2)}(\bar{x}^i(k)))^{-1}f^{(1)}(\bar{x}^i(k))$;
\For{$j=1$ \textbf{to} $i$}
\State $g_{j}(\bar{x}^i(k))=(\bar{R}+f^{(2)}(\bar{x}^i(k)))^{-1}[f^{(1)}(\bar{x}^i(k))+\bar{R}g_{j-1}(\bar{x}^i(k))];$
\EndFor
\State $\bar{x}^{i+1}(k)=\bar{x}^{i}(k)-g_{i}(\bar{x}^i(k));$
\EndFor
\State $\bar{x}(k)=\bar{x}^{I}(k)$;
\EndFor
	\end{algorithmic}
\end{algorithm}

\subsection{Extend to Output Feedback Control}
%The aforementioned section primarily considers the scenario where state information is fully available. In reality, it is more common for state information to be partially known.
The following will mainly explain how the above results can be extended to more general output feedback scenarios. Consider system (\ref{f2.1}) with the following output equations
\begin{eqnarray}
y_{i}(k)=H_{i}x(k),\ i=1,2. \label{o1}
\end{eqnarray}
In this case, our goal is to find the controllers with the following structures to minimize cost functional (\ref{f2.3}).
\begin{eqnarray}
u_{i}(k)=K^{y}_{i}(k)y_{i}(k).\label{o2}
\end{eqnarray}
For discussing, let $H=\begin{bmatrix}H_{1}&0\\0&H_{2}\end{bmatrix}, \ K^{y}(k)=\begin{bmatrix}K^{y}_{1}&0\\0&K^{y}_{2}\end{bmatrix}$.
In this case, we get
\begin{eqnarray}
X(k+1)=(A+BK^{y}H)X(k)(A+BK^{y}H)'.\label{s1}
\end{eqnarray}
Define the following Hamiltonian function
\begin{eqnarray}
\hspace{-3mm}&&\hspace{-3mm}H^{y}_{N}=tr[(Q+H'K^{y'}RK^{y}H)X(k)
+(A+BK^{y}H)X(k)(A+BK^{y}H)'P^{y}(k+1)]
\end{eqnarray}
%therefore,
%\begin{eqnarray}
%0\hspace{-3mm}&=&\hspace{-3mm}\frac{\partial H^{y}_{N}}{\partial K^{y}}=RK^{y}(k)HX(k)H'+B'P^{y}(k+1)(A+BK^{y}(k)H)X(k)H'\nonumber\\
%0\hspace{-3mm}&=&\hspace{-3mm}\frac{\partial H^{y}_{N}}{\partial X(k)}=
%Q+H'K^{y'}(k)RK^{y}(k)H
%+(A+BK^{y}(k)H)'P^{y}(k+1)(A+BK^{y}(k)H)
%\end{eqnarray}
Hence, we can obtain the expression for the gain matrices $K^{y}_{i}$ as follows.
\begin{small}\begin{eqnarray}
0\hspace{-3mm}&=&\hspace{-3mm}R_{11}K^{y}_{1}H_{1}X_{11}H_{1}'+R_{12}K^{y}_{2}H_{2}X'_{12}H_{1}'\nonumber\\
\hspace{-3mm}&&\hspace{-3mm}+[(B'_{11}P^{y}_{11}(k+1)+B'_{21}P^{y}_{12}(k+1))(A_{11}+B_{11}K^{y}_{1}H_{1})
+(B'_{11}P^{y}_{12}(k+1)+B'_{21}P^{y}_{22}(k+1))(A_{21}+B_{21}K^{y}_{1}H_{1})]X_{11}H_{1}'\nonumber\\
\hspace{-3mm}&&\hspace{-3mm}
+[(B'_{11}P^{y}_{11}(k+1)+B'_{21}P^{y}_{12}(k+1))(A_{12}+B_{12}K^{y}_{2}H_{2})
+(B'_{11}P^{y}_{12}(k+1)+B'_{21}P^{y}_{22}(k+1))(A_{22}+B_{22}K^{y}_{2}H_{2})]X'_{12}H_{1}'\nonumber\\
\hspace{-3mm}&=&\hspace{-3mm}\Upsilon^{y}_{11}K^{y}_{1}H_{1}X_{11}H_{1}'
+M^{y}_{11}(k)X_{11}(k)H_{1}'+\Upsilon_{12}^{y}(k)K^{y}_{2}(k)H_{2}X_{12}(k)'H_{1}'
+M^{y}_{12}(k)X_{12}(k)'H_{1}'
\nonumber\\
0\hspace{-3mm}&=&\hspace{-3mm}R'_{12}K^{y}_{1}H_{1}X_{12}H_{2}'+R_{22}K^{y}_{2}H_{2}X_{22}H_{2}'\nonumber\\
\hspace{-3mm}&&\hspace{-3mm}
+[(B'_{12}P^{y}_{11}(k+1)+B'_{22}P^{y'}_{12}(k+1))(A_{11}+B_{11}K^{y}_{1}H_{1})
+(B'_{12}P^{y}_{12}(k+1)+B'_{22}P^{y}_{22}(k+1))(A_{21}+B_{21}K^{y}_{1}H_{1})]X_{12}H_{2}'\nonumber\\
\hspace{-3mm}&&\hspace{-3mm}
+[(B'_{12}P^{y}_{11}(k+1)+B'_{22}P^{y'}_{12}(k+1))(A_{12}+B_{12}K^{y}_{2}H_{2})
+(B'_{12}P^{y}_{12}(k+1)+B'_{22}P^{y}_{22}(k+1))(A_{22}+B_{22}K^{y}_{2}H_{2})]X'_{22}H_{2}'\nonumber\\
\hspace{-3mm}&=&\hspace{-3mm}\Upsilon^{y'}_{12}K^{y}_{1}H_{1}X_{11}H_{2}'
+M^{y}_{21}(k)X_{12}(k)H_{2}'+\Upsilon_{22}^{y}(k)K^{y}_{2}(k)H_{2}X_{22}(k)H_{2}'
+M^{y}_{22}(k)X_{22}(k)'H_{2}',
\nonumber
\end{eqnarray}\end{small}
where
\begin{eqnarray*}
\Upsilon^{y}(k)&=&\begin{bmatrix}\Upsilon^{y}_{11}(k)&\Upsilon^{y}_{12}(k)\\ \Upsilon^{y}_{12}(k)'&\Upsilon^{y}_{22}(k)\end{bmatrix}=R+B'P^{y}(k+1)B,\label{o3.3}\\
M^{y}(k)&=&\begin{bmatrix}M^{y}_{11}(k)&M^{y}_{12}(k)\\ M^{y}_{21}(k)&M^{y}_{22}(k)\end{bmatrix}=B'P^{y}(k+1)A, \label{o3.4}
\end{eqnarray*}
with
\begin{eqnarray}
P^{y}(k)=Q+H'K^{y'}(k)RK^{y}(k)H
+(A+BK^{y}(k)H)'P^{y}(k+1)(A+BK^{y}(k)H).\label{o3.5}
\end{eqnarray}

\section{A Numerical Example}
Consider Problem 1 in the four-dimensional case with the following parameters.
\begin{eqnarray*}
A&=&\left[\begin{array}{cccc}
	1.1 & 0 & 0.1 & 0    \cr
	0.1 & 1.1 & 0 & 1    \cr
          0 & 0 & 1 & 0.1    \cr
          0 & 0.3 & 0 & 1.2
\end{array}\right],
B=\left[\begin{array}{cccc}
	1 & 1 & 0 & 2    \cr
	0 & 1 & 0 & 0    \cr
          0 & 0 & 2 & 0    \cr
          0 & 1 & 0 & 1
\end{array}\right],
Q=\left[\begin{array}{cccc}
	1 & 0 & 0 & 0    \cr
	0 & 1 & 0 & 0    \cr
          0 & 0 & 1 & 0    \cr
          0 & 0 & 0 & 1
\end{array}\right],\\
R&=&\left[\begin{array}{cccc}
	0.5 & 0 & 0 & 0    \cr
	0 & 0.5 & 0 & 0    \cr
          0 & 0 & 0.5 & 0    \cr
          0 & 0 & 0 & 0.5
\end{array}\right],
P(0)=\left[\begin{array}{cccc}
	2 & 0 & 0 & 0    \cr
	0 & 2 & 0 & 0    \cr
          0 & 0 & 2 & 0    \cr
          0 & 0 & 0 & 2
\end{array}\right],
x(0)=\left[\begin{array}{cccc}
	5     \cr
	3     \cr
          2     \cr
          3
\end{array}\right],
\end{eqnarray*}
and the number of experiments $I=90$, terminal time $N=50$.  Applying Algorithm 1, we calculate the solution of Problem 1, whose components as shown in Figure \ref{fig2} and Figure \ref{fig3}.
\begin{figure}[tbh]
\begin{centering}
\includegraphics[width=15cm]{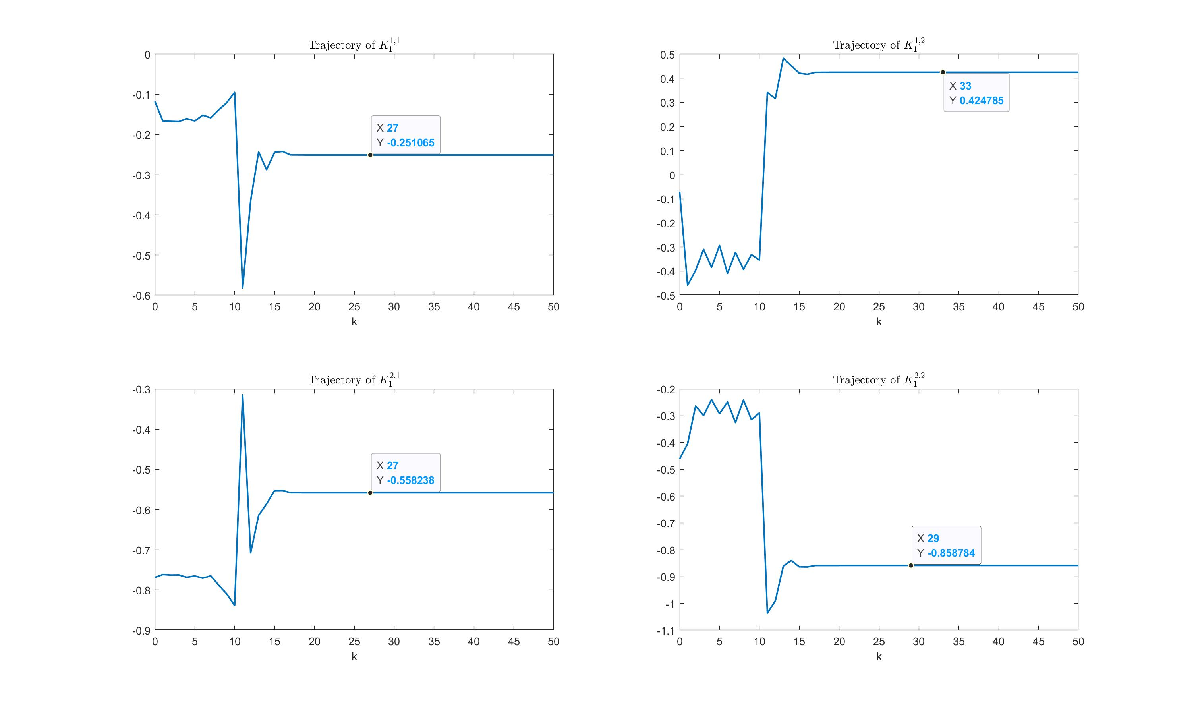}
\caption{\label{fig2}The optimal solution $K_{1}$}
\par\end{centering}
\end{figure}

\begin{figure}[tbh]
\begin{centering}
\includegraphics[width=15cm]{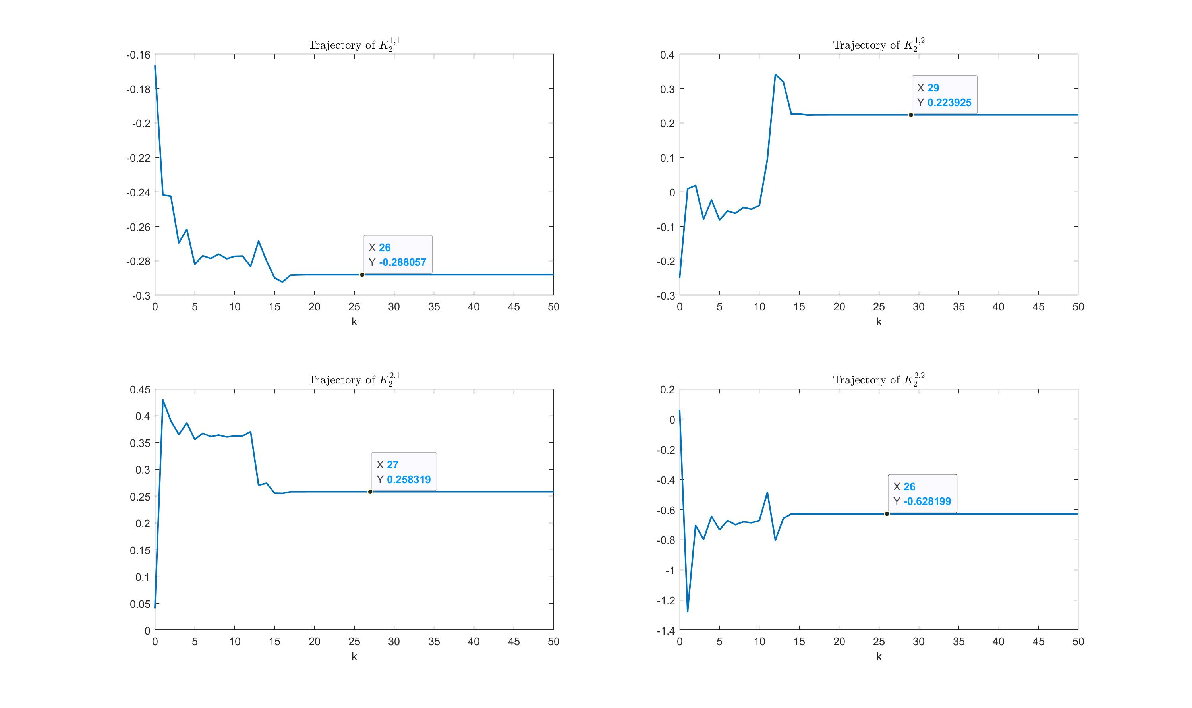}
\par\end{centering}
\caption{\label{fig3} The optimal solution $K_{2}$}
\end{figure}
And the error trajectory plot resulting from solving equations (\ref{3.1}) and (\ref{3.2}) is depicted in Figure \ref{fig1}. The calculation accuracy can reach up to $10^{-3}$,
demonstrating the effectiveness of the algorithm.
\begin{figure}[tbh]
\begin{centering}
\includegraphics[width=10cm]{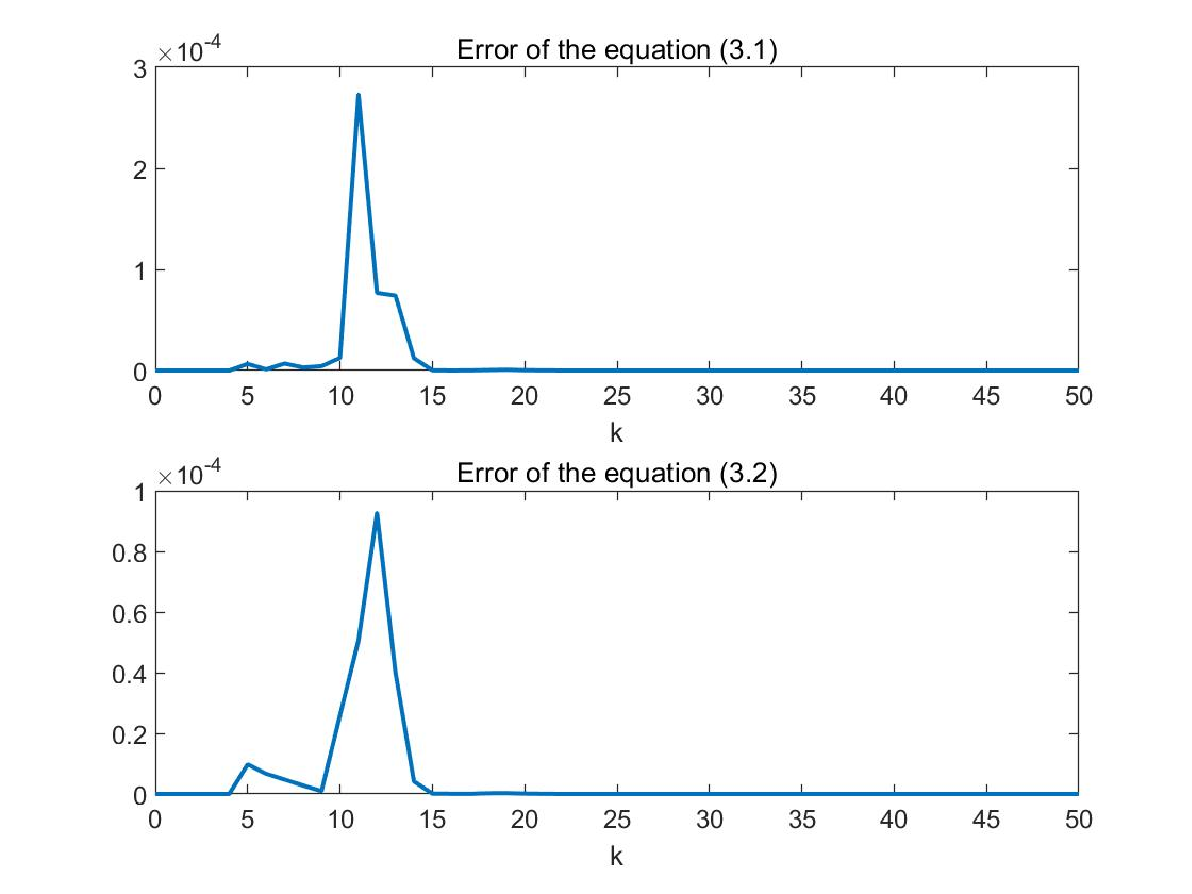}
\caption{\label{fig1} Calculating error of the equation}
\par\end{centering}
\end{figure}
Moreover, The state trajectory plot under the optimal solution is depicted in Figure \ref{fig4}.
\begin{figure}[tbh]
\begin{centering}
\includegraphics[width=8cm]{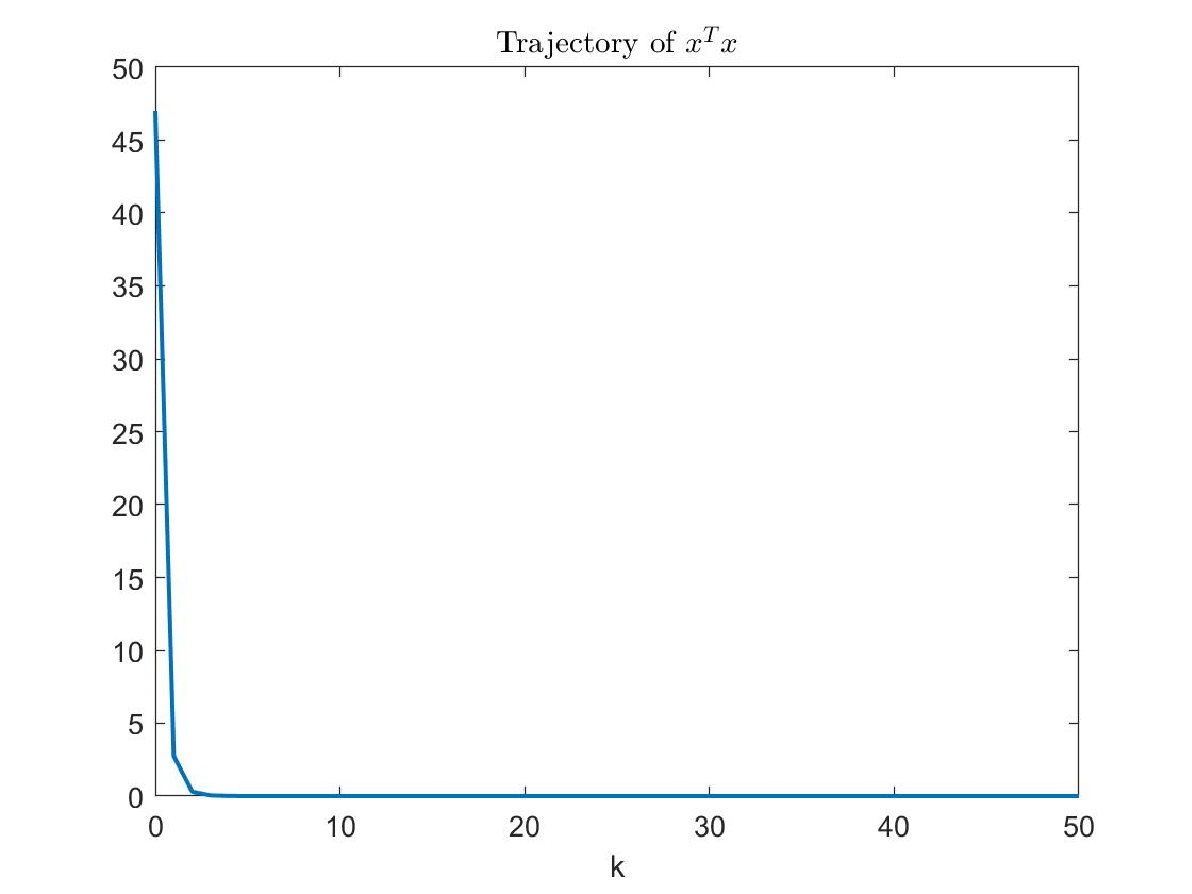}
\par\end{centering}
\caption{\label{fig4}: State trajectory }
\end{figure}

\section{Conclusion}
This paper has conducted an in-depth study on the decentralized control of linear deterministic systems where different controllers operate based on different state information, achieving the following major results. First, we proposed an expression for the controller gain matrices based on the matrix maximum principle, representing a significant innovation in the field of decentralized control. Second, by considering the algebraic Riccati equation in the steady-state case, we transformed the backward Riccati equation into a forward equation for solving, and applied a gradient descent algorithm to find the solution to the implicit equation. Simulation results indicate that our method effectively solves the problem of determining the controller gain matrices.

%\bibliographystyle{plain}        % Include this if you use bibtex
%\bibliography{autosam}           % and a bib file to produce the
                                 % bibliography (preferred). The
                                 % correct style is generated by
                                 % Elsevier at the time of printing.

%\begin{thebibliography}{99}     % Otherwise
%%
%%	\bibitem{juan}Z. Zhang, J. Xu and M. Fu, ``Q-Learning for feedback Nash strategy of finite-horizon nonzero-sum difference games", \emph{IEEE Transactions on Cybernetics}, doi: 10.1109/TCYB.2021.3052832.
%%	% Insert the full references here.
%	% See a recent issue of Automatica
%	% for the style.
%	
%	
%\end{thebibliography}

\end{document}